\documentclass[a4paper,fleqn,11pt]{amsart}

\usepackage{mathrsfs}
\usepackage{mathptmx}
\usepackage{cite}
\usepackage{amssymb}
\usepackage[all]{xy}

\DeclareMathAlphabet{\mathpzc}{OT1}{pzc}{m}{it}

\newtheorem{theorem}{Theorem}[section]
\newtheorem{lemma}[theorem]{Lemma}
\newtheorem{sublemma}[theorem]{Sublemma}

\newtheorem{proposition}[theorem]{Proposition}
\theoremstyle{definition}
\newtheorem{definition}[theorem]{Definition}

\newtheorem{notation}[theorem]{Notation}
\theoremstyle{remark}
\newtheorem{remark}[theorem]{Remark}
\newtheorem{claim}[theorem]{Claim}

\numberwithin{equation}{section}

\allowdisplaybreaks[4]

\begin{document}

\title{Zariski-van Kampen theorems for singular varieties\,---\,an approach via the relative monodromy variation}

\author{Christophe Eyral and Peter Petrov}

\address{C. Eyral, Institute of Mathematics, Polish Academy of Sciences, \'Sniadeckich~8, 00-656 Warsaw, Poland}  
\email{eyralchr@yahoo.com} 
\address{P. Petrov, Universidade Federal Fluminense, Av. dos Trabalhadores, 420, Vila Santa Cecilia, 27225-125, Volta Redonda, RJ, Brazil, }
\email{pk5rov@gmail.com}


\subjclass[2010]{14F35, 14D05, 32S50.}
\keywords{Fundamental group; singular (quasi-projective) variety; pencil of hyperplane sections; relative loop; relative monodromy variation.}

\begin{abstract}
The classical Zariski-van Kampen theorem gives a presentation of the fundamental group of the complement of a complex algebraic curve in $\mathbb{P}^2$. The first generalization of this theorem to singular (quasi-projective) varieties was given by the first author. In both cases, the relations are generated by the standard monodromy variation operators associated with the special members of a generic pencil of hyperplane sections. 
In the present paper, we give a new generalization in which the relations are generated by the \emph{relative} monodromy variation operators introduced by D. Ch\'eniot and the first author. The advantage of using the relative operators is not only to cover a larger class of varieties but also to unify the Zariski-van Kampen type theorems for the fundamental group and for higher homotopy groups. In the special case of non-singular varieties, the main result of this paper was conjectured by D.~Ch\'eniot and the first author.
\end{abstract}

\maketitle

\markboth{C. Eyral and P. Petrov}{Zariski-van Kampen theorems for singular varieties}  

\section{Introduction}

Let $X:=Y\setminus Z$ be a (possibly singular) quasi-projective variety in the complex projective space $\mathbb{P}^n$ and let $L$ be a generic hyperplane of $\mathbb{P}^n$. By the singular versions of the Lefschetz hyperplane section theorem (cf.~\cite{E2,HL,GM1,GM2}), we know that there exists an integer $q(Y,Z)$---depending on the nature and the position of the singularities of $Y$ and $Z$---such that the pair $(X,L\cap X)$ is $q(Y,Z)$-connected. For instance, if $X$ is a purely dimensional non-singular or local complete intersection variety, then $q(Y,Z)=\dim X-1$ (cf.~\cite{HL2}). In the present paper, we are interested in the special class of varieties for which the integer $q(Y,Z)$ is equal to $1$. (In particular, this includes non-singular and local complete intersection varieties of pure dimension $2$.) For such a variety, the natural map
\begin{equation*}
\pi_q(L\cap X,x_0) \to \pi_q(X,x_0)
\end{equation*}
is bijective for $q=0$ and surjective for $q=1$, where $\pi_1(\cdot)$ denotes the fundamental group and $\pi_0(\cdot)$ the set of path-connected components. Our goal is to determine the kernel of the map $\pi_1(L\cap X,x_0) \to \pi_1(X,x_0)$. In the special case where $X=\mathbb{P}^2\setminus C$ with $C$ an algebraic curve, O. Zariski \cite{Z} and E. K. van Kampen \cite{vK} showed that the kernel in question is generated by the standard monodromy variation operators associated with the ``special'' members of a generic pencil of line sections. (By special sections, we mean those arising from the lines of the pencil which are tangent to the curve or that cross a singularity.) Thus, combined with the surjectivity of the map $\pi_1(L\cap (\mathbb{P}^2\setminus C),x_0) \to \pi_1(\mathbb{P}^2\setminus C,x_0)$, the fundamental group $\pi_1(\mathbb{P}^2\setminus C,x_0)$ is the quotient of $\pi_1(L\cap (\mathbb{P}^2\setminus C),x_0)$ by the ``monodromy relations''---that is, by the normal subgroup of $\pi_1(L\cap (\mathbb{P}^2\setminus C),x_0)$ generated by the standard monodromy variation operators. 

The first generalization of the Zariski-van Kampen theorem to singular varieties was given by the first author in \cite{E1}. There, as in the case of plane curve complements, the kernel of the map $\pi_1(L\cap X,x_0) \to \pi_1(X,x_0)$ is generated by the standard monodromy variation operators associated with the special members of a generic pencil of hyperplane sections. For example, if $X$ is the complement of a curve $C$ in a surface $S$ of $\mathbb{P}^3$ with $S\setminus C$ non-singular, then the result can be easily stated as follows. Consider a generic pencil $\Pi$ of hyperplanes of $\mathbb{P}^3$ such that $L\in\Pi$. Write $\Pi_0$ for its base locus (i.e., $\Pi_0$ is the $(n-2)$-plane given by the intersection of all the members of $\Pi$), and assume that the natural map 
\begin{equation}\label{hypothesis}
\pi_0(\Pi_0\cap (S\setminus C))\rightarrow
\pi_0(L\cap (S\setminus C)) 
\end{equation}
is bijective. (Note that when $S=\mathbb{P}^2$ this condition is always satisfied.) Under these assumptions, Theorem 5.1 or Corollary 5.3 of \cite{E1} says that if $x_0$ is a base point in $\Pi_0\cap (S\setminus C)$, then the fundamental group $\pi_1(S\setminus C,x_0)$  is the quotient of the group $\pi_1(L\cap (S\setminus C),x_0)$ by the monodromy relations---that is, all the relations of the form
\begin{align}\label{relabsolute}
\mbox{Var}_h([\alpha]):=[\alpha]^{-1} h_\#([\alpha])=
[\alpha^{-1}\cdot h\circ\alpha]=e,
\end{align}
where $[\alpha]\in \pi_1(L\cap (S\setminus C),x_0)$, $e$ is the trivial element, and $h_\#$ is the homomorphism induced in homotopy by a monodromy $h$ associated with a special hyperplane of the pencil. Note that $h$ can always be chosen so that it is the identity on $\Pi_0 \cap (S\setminus C)$, and hence the composition of loops in (\ref{relabsolute}) is well defined.

A conjecture of D. Ch\'eniot and the first author \cite[\S 4]{CE} says that the above mentioned result still holds true when the map (\ref{hypothesis}) is not bijective provided that we consider the action of the monodromies not only on the absolute loops of $L\cap (S\setminus C)$ but also on the ``relative'' loops of $L\cap (S\setminus C)$ modulo $\Pi_0\cap (S\setminus C)$. (Here, by a relative loop, we mean a path $\alpha\colon I:=[0,1]\rightarrow L\cap (S\setminus C)$ with $\alpha(1)=x_0$---the base point---and $\alpha(0)\in \Pi_0\cap (S\setminus C)$.) Precisely, as $h$ is the identity on $\Pi_0\cap (S\setminus C)$, the composition  
\begin{align*}
\alpha^{-1} \cdot h\circ\alpha\colon I \rightarrow L\cap (S\setminus C)
\end{align*}
of the relative loops $\alpha^{-1}$ and $h\circ\alpha$, defined by
\begin{align}\label{comprelloop}
\alpha^{-1} \cdot h\circ\alpha(t): = 
\left\{
\begin{aligned}
& \alpha(1-2t)&\mbox{ for }\quad  0\leq t\leq 1/2,\\
& h\circ\alpha(2t-1) &\mbox{ for } \quad 1/2\leq t\leq 1,
\end{aligned}
\right.
\end{align}
is well defined and is an absolute loop based at $x_0$ (i.e., $\alpha^{-1} \cdot h\circ\alpha(0)=\alpha^{-1} \cdot h\circ\alpha(1)=x_0$), even if $\alpha$ is a relative loop. Note that, in general, composing relative loops does not make sense.  The possibility to perform such a composition in our situation comes from the fact that $h$ is the identity on $\Pi_0\cap (S\setminus C)$. Observe that when $\alpha$ is an absolute loop, the relation (\ref{comprelloop}) is nothing but the standard composition of the absolute loops $\alpha^{-1}$ and $h\circ\alpha$. The conjecture in \cite[\S 4]{CE}  says that the result about $\pi_1(S\setminus C,x_0)$ which we have mentioned above still holds true when the map (\ref{hypothesis}) is not bijective provided  that we add all the relations of the form
\begin{align*}\label{relrelative}
\mbox{Var}^{\, \mbox{\tiny rel}}_h([\alpha]):=[\alpha^{-1}\cdot h\circ\alpha]=e
\end{align*}
to the relations (\ref{relabsolute}), with this time $[\alpha]$ belonging to the relative homotopy \emph{set} $\pi_1(L\cap (S\setminus C),\Pi_0\cap (S\setminus C),x_0)$. (Note that, in general, this pointed set does not have a group structure.)

In the present paper, we prove that this conjecture is true. Actually, we prove a new generalization of the Zariski-van Kampen theorem which not only implies the above conjecture (i.e., the non-singular case) but which also covers a class of \emph{singular} varieties larger than the class covered by Theorem 5.1 of \cite{E1}. Besides to obtain a larger class of varieties, another advantage of using the relative operators is to unify the Zariski-van Kampen type theorems for the fundamental group and for higher homotopy groups. (For details about higher homotopy groups, we refer the reader to \cite{Lib1,CL,CE}.)

\begin{notation}
Throughout, $I$ denotes the unit interval $[0,1]$.
If $(A,B)$ is a pointed pair with base point $b\in B$, we denote by $F^1(A,B,b)$ the set of \emph{relative} loops of $A$ modulo $B$ based at $b$. These are (continuous) maps $\alpha\colon I\to A$ such that $\alpha(0)\in B$ and $\alpha(1)=b$.
We denote by $F^1(A,b)$ the set of loops of $A$ based at $b$---that is, maps $\alpha\colon I\to A$ such that $\alpha(0)=\alpha(1)=b$. We sometimes say \emph{absolute} loop instead of loop to emphasize the contrast with relative loops. 

Given $\alpha$ in $F^1(A,B,b)$ (respectively, in $F^1(A,b)$), we denote by $[\alpha]_{A,B,b}$ (respectively, by $[\alpha]_{A,b}$) the homotopy class of $\alpha$ in the pointed set $\pi_1(A,B,b)$ (respectively, in the fundamental group $\pi_1(A,b)$). When there is no ambiguity, we omit the subscripts. If $[\alpha]_{A,B,b}=[\beta]_{A,B,b}$ (respectively, $[\alpha]_{A,b}=[\beta]_{A,b}$), then we use the expression ``$\alpha$ and $\beta$ are homotopic in $(A,B,b)$ (respectively, in $(A,b)$)''. We write $e$ for the trivial element of the group $\pi_1(A,b)$ (i.e., the homotopy class of the constant loop based at $b$). As usual, $\pi_0(A)$ will denote the set of path-connected components of $A$.

For any map $g\colon (A,b)\to (A',b')$ of pointed sets (i.e., $g(b)=b'$), we denote by $g_\#\colon \pi_1(A,b)\to \pi_1(A',b')$ the homomorphism induced by $g$. 
By a \emph{natural} map, we mean the homomorphism induced by an inclusion map. 

Finally, we use standard notation from homotopy theory. For example, if $\alpha$ and $\beta$ are paths in $A$ with $\alpha(1)=\beta(0)$, then we write $\alpha\cdot\beta$ (or simply $\alpha\beta$) for the composition or product of $\alpha$ and $\beta$, which is defined by
\begin{align*}
\alpha \cdot \beta(t): = 
\left\{
\begin{aligned}
& \alpha(2t)&\mbox{ for }\quad  0\leq t\leq 1/2,\\
& \beta(2t-1) &\mbox{ for } \quad 1/2\leq t\leq 1;
\end{aligned}
\right.
\end{align*}
we denote by $\alpha^{-1}$ the ``inverse path'' to $\alpha$, which is defined by 
$\alpha^{-1}(t):=\alpha(1-t)$;
and so on. 

For further details in homotopy theory, we refer the reader 
for instance to \cite[\S 15]{St}.
\end{notation}

\section{Standard and relative monodromy variation operators}\label{sect-varop}

Let $X:=Y\setminus Z$ be a quasi-projective variety in $\mathbb{P}^n$ ($n\geq 2$)---that is, $Y$ is a non-empty closed algebraic subset of $\mathbb{P}^n$ and $Z$ is a proper closed algebraic subset of $Y$. Pick a Whitney stratification $\Xi$ of $Y$ such that $Z$ is a union of strata, and consider a projective hyperplane $L$ of $\mathbb{P}^n$ transverse to (the strata of) $\Xi$. (The choice of such a hyperplane is generic.) Then choose a pencil $\Pi$ of hyperplanes of $\mathbb{P}^n$ so that its base locus $\Pi_0$---which is also called the \emph{axis} of $\Pi$---is transverse to $\Xi$ and such that $L\in\Pi$. (The choice of such an $(n-2)$-plane $\Pi_0$ of $\mathbb{P}^n$ is generic inside the hyperplane $L$.) Then all the members of $\Pi$ are transverse to $\Xi$ except a finite number of them $L_1,\ldots,L_N$---so-called \emph{special} hyperplanes of $\Pi$. Furthermore, for each $L_i$ ($1\leq i\leq N$), there is only a finite number of points where $L_i$ is not transverse to $\Xi$. Let us denote by $\Sigma_i$ the set of such points, and let us write 
\begin{equation*}
\Sigma:=\bigcup_{1\leq i\leq N}\Sigma_i.
\end{equation*}
It is worth to observe that the intersection $\Sigma\cap\Pi_0$ is empty. Also, note that if $L'$ is not a special hyperplane of $\Pi$, then the pair $(L'\cap X,\Pi_0\cap X)$ is homeomorphic to the pair $(L\cap X,\Pi_0\cap X)$. For details we refer the reader to \cite{C4}.

Now parametrize the elements of $\Pi$ by $\mathbb{P}^1$ as usual, and write $\lambda$ (respectively,~$\lambda_i$) for the parameter corresponding to the generic hyperplane $L$ (respectively, to the special hyperplane $L_i$). For each $1\leq i\leq N$, pick a small closed disc $D_i\subseteq\mathbb{P}^1$ centred at $\lambda_i$ and fix a point $\ell_i$ on its boundary $\partial D_i$. Choose the $D_i$'s mutually disjoint. Finally, take a simple path $\rho_i$ in $\mathbb{P}^1$ joining $\lambda$ to $\ell_i$ so that:
\begin{enumerate}
\item
$\mbox{im}(\rho_i)\cap D_i=\{\ell_i\}$;
\item
$\mbox{im}(\rho_i)\cap \mbox{im}(\rho_j)=\{\lambda\}$ if $i\not=j$;
\item
$\mbox{im}(\rho_i)\cap D_j=\emptyset$ if $i\not=j$.
\end{enumerate}

\begin{notation}
For any subsets $E\subseteq\mathbb{P}^n$ and $\Lambda\subseteq\mathbb{P}^1$, we set
\begin{equation*}
E_\Lambda:=\bigcup_{\ell\in \Lambda} (\Pi(\ell)\cap E),
\end{equation*}
where $\Pi(\ell)$ is the member of $\Pi$ with parameter $\ell$. For example, $X_{\{\lambda\}}=\Pi(\lambda)\cap X=L\cap X$. Hereafter, to simplify, we shall write $X_\lambda$ instead of $X_{\{\lambda\}}$.
Also, because of a frequent use of the section of $X$ by the base locus $\Pi_0$ of the pencil, we set
\begin{equation*}
A:=\Pi_0\cap X.
\end{equation*}
Finally, throughout we shall write
\begin{equation*}
\mathbb{P}^{1*}:=\mathbb{P}^{1}\setminus\{\lambda_1,\ldots,\lambda_N\}.
\end{equation*}
\end{notation}

\subsection{Monodromy}
For each $1\leq i\leq N$, set $K_i:=\mbox{im}(\rho_i)\cup D_i$, choose a loop $\delta_i\colon I\to\partial D_i$ which runs once counterclockwise along the boundary of $D_i$, starting and ending at $\ell_i$, and consider the loop $\omega_i\colon I\to\partial K_i$ along the boundary $\partial K_i$ of $K_i$ defined by the composition 
\begin{equation*}
\omega_i:=\rho_i\delta_i\rho_i^{-1}.
\end{equation*}
(In particular, we have $\omega_i(0)=\omega_i(1)=\lambda$.)
By \cite[Lemma 4.1]{C1}, for each $i$, there is an isotopy
\begin{equation}\label{isotopyHi}
H_i\colon X_\lambda \times I \to X_{\partial K_i},
\quad (x,\tau)\mapsto H_i(x,\tau),
\end{equation}
satisfying the following properties:
\begin{enumerate}
\item
$H_i(x,0)=x$ for any $x\in X_\lambda$;
\item 
$H_i(x,\tau)\in X_{\omega_i(\tau)}$ for any $x\in X_\lambda$ and any $\tau\in I$;
\item
for each $\tau\in I$, the map $X_\lambda \to X_{\omega_i(\tau)}$, defined by $x\mapsto H_i(x,\tau)$, is a homeomorphism;
\item 
$H_i(x,\tau)=x$ for any $x\in A$ and any $\tau\in I$.
\end{enumerate}
The terminal homeomorphism $h_i\colon X_\lambda\to X_\lambda$ of the above isotopy, defined by
\begin{equation*}
x\mapsto h_i(x):=H_i(x,1)
\end{equation*}
leaves $A$ pointwise fixed. 

\begin{definition}
The map $h_i$ is called a \emph{monodromy} of $X_\lambda$ relative to $A$ above $\omega_i$. 
\end{definition}

\begin{remark}[\mbox{cf.~\cite[Lemma 4.3]{C1}}]\label{dohcoi}
Another choice of $\omega_i$ within the same homotopy class $[\omega_i]\in \pi_1(\mathbb{P}^{1*},\lambda)$ and another choice of $H_i$ as above would give a new monodromy isotopic to $h_i$ within $X_\lambda$ by an isotopy leaving $A$ pointwise fixed. In other words, the isotopy class of $h_i$ in $X_\lambda$ relative to $A$ is completely determined by $[\omega_i]$.
\end{remark}

\subsection{Standard monodromy variation operator}
We assume that $A\not=\emptyset$, and we fix a base point $x_0\in A$. 
As $h_i$ leaves $x_0$ fixed, it induces an automorphism
\begin{equation*}
h_{i\#}\colon \pi_1(X_\lambda,x_0) \overset{\sim}\rightarrow 
\pi_1(X_\lambda,x_0).
\end{equation*}
This automorphism in turn induces a map
\begin{align*}
\mbox{Var}_{h_i}\colon \pi_1(X_\lambda,x_0) & \to \pi_1(X_\lambda,x_0)\\ [\alpha] & \mapsto [\alpha]^{-1} h_{i\#}([\alpha])=[\alpha^{-1}\cdot h_i\circ\alpha]
\end{align*}
which only depends on the homotopy class $[\omega_i]\in \pi_1(\mathbb{P}^{1*},\lambda)$.

\begin{definition}
The map $\mbox{Var}_{h_i}$ is the standard monodromy variation operator associated to $[\omega_i]$.
\end{definition}

\subsection{Relative monodromy variation operator}
This operator was introduced by D. Ch\'eniot and the first author in \cite[\S 4]{CE}. Again let $x_0$ be a base point in $A\not=\emptyset$.  Pick a \emph{relative} loop $\alpha\in F^1(X_\lambda,A,x_0)$, and consider the map $\alpha^{-1}\cdot h_i\circ\alpha$ defined on~$I$ by
\begin{align}\label{comprelloop-2}
\alpha^{-1}\cdot h_i\circ\alpha(t):= 
\left\{
\begin{aligned}
& \alpha^{-1}(2t):=\alpha(1-2t)&\mbox{ for }\quad  0\leq t\leq 1/2,\\
& h_i\circ\alpha(2t-1) &\mbox{ for } \quad 1/2\leq t\leq 1.
\end{aligned}
\right.
\end{align}
As $h_i$ is the identity on $A$, this map is well defined and belongs to $F^1(X_\lambda,x_0)$, that is, $\alpha^{-1}\cdot h_i\circ\alpha$ is an \emph{absolute} loop.
(Again observe that in the special case where $\alpha$ is an absolute loop, the relation (\ref{comprelloop-2}) is nothing but the standard composition of the absolute loops $\alpha^{-1}$ and $h\circ\alpha$.) By \cite[Lemma 4.1]{CE}, the correspondence
\begin{align*}
\mbox{Var}^{\, \mbox{\tiny rel}}_{h_i}\colon \pi_1(X_\lambda,A,x_0) & \to \pi_1(X_\lambda,x_0)\\ 
[\alpha]_{X_\lambda,A,x_0} & \mapsto 
[\alpha^{-1}\cdot h_i\circ\alpha]_{X_\lambda,x_0}
\end{align*}
is well defined and only depends on the homotopy class $[\omega_i]\in \pi_1(\mathbb{P}^{1*},\lambda)$. 

\begin{definition}
The map $\mbox{Var}^{\, \mbox{\tiny rel}}_{h_i}$ is the \emph{relative} monodromy variation operator associated to $[\omega_i]$. (In \cite{CE}, the map $\mbox{Var}^{\, \mbox{\tiny rel}}_{h_i}$ is denoted by $\mbox{VAR}_{i,1}$.)
\end{definition}

\begin{remark}
Relative variation operators can also be defined for higher homotopy groups. For details we refer the reader to \cite[\S 4]{CE}.
\end{remark}

\section{Statements of the main results}\label{sect-smr}

By \cite[Th\'eor\`eme 2.5]{E2}, if $(X,X\setminus\Sigma)$ is $1$-connected\footnote{In the terminology of \cite{E2}, the $1$-connectivity of the pair $(X,X\setminus\Sigma)$ corresponds to the assumption that the global rectified homotopical depth of $X$ along $\Sigma\cap X$ is greater than or equal to $2$. This assumption is a measure of the degree of singularity of $X$. For example, \cite[Corollary 3.2.2]{HL2} and \cite[Th\'eor\`eme 3.11]{E4} show that if $X$ is locally a complete intersection of pure dimension $2$, then the global rectified homotopical depth of $X$ along $\Sigma\cap X$ is greater than or equal to $2$.} and $(X_\lambda, A)$ is $0$-connected, then $(X,X_\lambda)$ is $1$-connected. In particular, the set $A$ is not empty, and for any base point $x_0\in A$, the natural map
\begin{equation*}
\pi_q(X_\lambda,x_0) \to \pi_q(X,x_0)
\end{equation*}
is bijective for $q=0$ and surjective for $q=1$. Concerning the kernel of $\pi_1(X_\lambda,x_0) \to \pi_1(X,x_0)$, the first author also obtained---in \cite{E1}---the following result which was the first generalization of the Zariski-van Kampen theorem to singular (quasi-projective) varieties.

\begin{theorem}[cf.~{\cite[Theorem 5.1]{E1}}]\label{Eyral-1}
Assume that $X\setminus \Sigma\not=\emptyset$. Furthermore, suppose that for any base point $y_0\in X\setminus \Sigma$, the following three conditions hold:
\begin{enumerate}
\item
$\pi_q(X\setminus \Sigma,y_0) \to \pi_q(X,y_0)$ is bijective for $q\in\{0,1\}$;\footnote{For $q=0$, the bijectivity of $\pi_0(X\setminus \Sigma,y_0) \to \pi_0(X,y_0)$ is equivalent to that of $\pi_0(X\setminus \Sigma) \to \pi_0(X)$. Indeed, by definition, $\pi_0(X\setminus \Sigma,y_0)$ is nothing but the \emph{pointed} set of the path-connected components of $X\setminus \Sigma$, where the ``point'' is the path-connected component containing $y_0$. Similarly for $\pi_0(X,y_0)$.}
\item
$\pi_0(A) \to \pi_0(X_\lambda)$ is bijective;
\item
$\pi_0(A) \to \pi_0(X_{\lambda_i}\setminus \Sigma_i)$ is surjective for all $1\leq i\leq N$.
\end{enumerate}
Then, for any base point $x_0\in A$, the map
\begin{equation*}
\pi_q(X_\lambda,x_0) \to \pi_q(X,x_0)
\end{equation*}
is bijective for $q=0$ and surjective for $q=1$ (as observed above under weaker assumptions); moreover, the kernel of $\pi_1(X_\lambda,x_0) \to \pi_1(X,x_0)$ coincides with the normal subgroup $\overline{\bigcup_{1\leq i\leq N} \mbox{\emph{im}}(\mbox{\emph{Var}}_{h_i})}$
of $\pi_1(X_\lambda,x_0)$ generated by the union of the images of the standard monodromy variation operators. In particular, for any $x_0\in A$, there is an isomorphism
\begin{equation*}
\pi_1(X_\lambda,x_0)\bigg/ \overline{\bigcup_{1\leq i\leq N} \mbox{\emph{im}}(\mbox{\emph{Var}}_{h_i})} \overset{\sim}{\longrightarrow} \pi_1(X,x_0).
\end{equation*}
\end{theorem}

In this theorem, all the maps are induced by inclusions. 

Under the same assumptions, the standard operators $\mbox{Var}_{h_i}$ can be replaced by the relative operators $\mbox{Var}^{\, \mbox{\tiny rel}}_{h_i}$ (cf.~\cite[\S 6]{E1}). Note that, in general, the normal subgroup generated by the images of the standard operators is (a priori) smaller than the normal subgroup generated by the images of the relative operators.

The first and the second assumptions in Theorem \ref{Eyral-1} are just natural extensions of the hypotheses of \cite[Th\'eor\`eme 2.5]{E2} that we have mentioned above.
Indeed, a pair $(U,V)$ of topological spaces (where $V$ is a non-empty subspace of $U$) is $q_0$-connected if and only if for every base point $v\in V$ the natural map $\pi_q(V,v) \to \pi_q(U,v)$ is bijective for $0\leq q\leq q_0-1$ and surjective for $q=q_0$. 

In the special case of \emph{non-singular} varieties, the statements of the above mentioned results (i.e., \cite[Th\'eor\`eme 2.5]{E2} and \cite[Theorem 5.1]{E1}) can be simplified. Indeed, by \cite[Th\'eor\`eme 4.3]{E3}, if $X$ is non-singular, then the pair $(X,X\setminus\Sigma)$ is $(2d-1)$-connected, where $d$ is the smallest dimension of the irreducible components of $Y$ not contained in $Z$. Therefore, if furthermore $d\geq 1$, then Th\'eor\`eme 2.5 of \cite{E2} simplifies as follows: ``If $X$ is non-singular, if $d\geq 1$ and if the pair $(X_\lambda, A)$ is $0$-connected, then the pair $(X,X_\lambda)$ is $1$-connected''. Note that if in addition $d\geq 2$, then $X_\lambda$ and $A$ are non-singular too and the smallest dimension of the irreducible components of $Y_\lambda$ not contained in $Z_\lambda$ is greater than or equal to $1$ (cf.~\cite[Lemme 11.3]{C4}), and hence, by the non-singular version of the Lefschetz hyperplane section theorem (cf.~\cite{E2,GM1,GM2,HL}), the pair $(X_\lambda,A)$ is always $0$-connected. Moreover, in this case, Theorem \ref{Eyral-1} can also be simplified as follows.

\begin{theorem}[{cf.~\cite[Corollary 5.3]{E1}}]\label{Eyral-2}
Assume that $X$ is non-singular and $d\geq 2$. Under these assumptions, if furthermore $\pi_0(A) \to \pi_0(X_\lambda)$ is injective (and hence bijective), then, for any base point $x_0\in A$, there is an isomorphism
\begin{equation}\label{isozvknsv}
\pi_1(X_\lambda,x_0)\bigg/ \overline{\bigcup_{1\leq i\leq N} \mbox{\emph{im}}(\mbox{\emph{Var}}_{h_i})} \overset{\sim}{\longrightarrow} \pi_1(X,x_0).
\end{equation}
\end{theorem}

\begin{remark}
In \cite[Corollary 5.3]{E1}, it is assumed that the maps $\pi_0(A)\rightarrow\pi_0(X_{\lambda_i}\setminus\Sigma_i)$ are surjective for all $1\leq i\leq N$. Actually, this assumption is redundant (it is always satisfied). Indeed, when $X$ is non-singular and $d\geq 2$, the subset $X_{\lambda_i}\setminus\Sigma_i$ is non-singular and the smallest dimension of the irreducible components of $Y_{\lambda_i}$ not contained in $(Z\cup\Sigma_i)_{\lambda_i}$ is greater than or equal to $1$ (cf.~\cite[Lemme 11.3]{C4}). Therefore, by the non-singular version of the Lefschetz hyperplane section theorem, the pair $(X_{\lambda_i}\setminus\Sigma_i,A)$ is always $0$-connected.
\end{remark}

\begin{remark}
In \cite{S1,S2}, I.~Shimada proves isomorphism (\ref{isozvknsv}) of Theorem \ref{Eyral-2} (i.e., the non-singular case) under different assumptions. Namely, he supposes that $X$ is connected and that the special sections $X_{\lambda_i}$ ($1\leq i\leq N$) are \emph{irreducible}. In particular, \cite{S1,S2} already contain Theorem \ref{Eyral-2} in the special case where the non-singular variety $X$ is connected.
\end{remark}

Note that, by the non-singular version of the Lefschetz hyperplane section theorem, if $X$ is non-singular, then the pair $(X,X_\lambda)$ is $(d-1)$-connected. In particular, if furthermore $d\geq 3$, then there is an isomorphism $\pi_1(X_\lambda,x_0)\overset{\sim}{\rightarrow} \pi_1(X,x_0)$. Thus Theorem \ref{Eyral-2} gives new information only in the case where $d=2$. Also, observe that in the special case where $Y=\mathbb{P}^2$ and $Z$ is an algebraic curve, the spaces $A$ and $X_\lambda$ are path-connected, and hence the map $\pi_0(A) \to \pi_0(X_\lambda)$ is always injective. Thus, in this case, Theorem \ref{Eyral-2} reduces to the classical Zariski-van Kampen theorem \cite{vK,Z}.

Now, if the map $\pi_0(A) \to \pi_0(X_\lambda)$ involved in Theorem \ref{Eyral-2} is \emph{not} injective, then, in general, without any further assumption, it seems that the normal subgroup generated by the standard monodromy variation operators is not big enough to contain the entire kernel of the map $\pi_1(X_\lambda,x_0) \to \pi_1(X,x_0)$. (However this should be contrasted with \cite[Theorem 3.2]{Libgober}.) Suppose, for instance, that $X_\lambda$ has a (path-connected) component $X_\lambda'$ that contains at least two components $A_0$ and $A_1$ of $A$. Pick points $x_0\in A_0$ and $x_1\in A_1$. By \cite[Lemma 4.8]{CE}, if $\alpha$ is any relative loop of $F^1(X_\lambda,A,x_0)$ such that $\alpha(0)=x_1$, then, for each $1\leq i\leq N$, its variation $\mbox{Var}^{\, \mbox{\tiny rel}}_{h_i}([\alpha])$ must be in the kernel of the map $\pi_1(X_\lambda,x_0) \to \pi_1(X,x_0)$. It is unclear to the authors whether or not $\mbox{Var}^{\, \mbox{\tiny rel}}_{h_i}([\alpha])$ can always be described in terms of (a product of possibly negative powers of) the standard monodromy variation operators.  In \cite{CE}, D.~Ch\'eniot and the first author conjectured that the conclusion of Theorem \ref{Eyral-2} still hold true---even when the map $\pi_0(A) \to \pi_0(X_\lambda)$ is \emph{not} injective---provided that we replace the standard monodromy variation operators by the relative ones. 
(In fact, the conjecture in \cite{CE} is much more general, as it also includes a similar statement for higher homotopy groups.)
Our first result says that (the $\pi_1$ part of) this conjecture is true. More precisely, we have the~following~statement.

\begin{theorem}\label{fmt}
Assume that $X$ is non-singular and $d\geq 2$. Under these assumptions, for any base point $x_0\in A$, there is an isomorphism
\begin{equation*}
\pi_1(X_\lambda,x_0)\bigg/ \overline{\bigcup_{1\leq i\leq N} \mbox{\emph{im}}(\mbox{\emph{Var}}^{\, \mbox{\tiny \emph{rel}}}_{h_i})} 
\overset{\sim}{\longrightarrow} \pi_1(X,x_0),
\end{equation*}
where $\overline{\bigcup_{1\leq i\leq N} \mbox{\emph{im}}(\mbox{\emph{Var}}^{\, \mbox{\tiny \emph{rel}}}_{h_i})}$ is the normal subgroup of $\pi_1(X_\lambda,x_0)$ generated by the union of the images of the relative monodromy variation operators. 
\end{theorem} 

Note that in the special case where $Y=\mathbb{P}^2$ and $Z$ is an algebraic curve, the set $A=\Pi_0\cap (\mathbb{P}^2\setminus Z)$ reduces to $\{x_0\}$ and the relative homotopy set $\pi_1(X_\lambda, \{x_0\},x_0)$ is nothing but the fundamental group $\pi_1(X_\lambda,x_0)$, so that in this case the operators $\mbox{Var}_{h_i}$ and $\mbox{Var}_{h_i}^{\, \mbox{\tiny rel}}$ coincide. In other words, when $Y=\mathbb{P}^2$ and $Z$ is an algebraic curve, Theorem \ref{fmt} also reduces to the classical Zariski-van Kampen theorem.

In fact, Theorem \ref{fmt} is an immediate corollary of our main result, which includes \emph{singular} varieties. Here is the precise statement.

\begin{theorem}\label{mt}
Assume that $X\setminus \Sigma\not=\emptyset$. Furthermore, suppose that for any base point $y_0\in X\setminus \Sigma$, any $q\in\{0,1\}$ and any integer $1\leq i\leq N$, the map
\begin{equation*}
\pi_q(X\setminus \Sigma,y_0) \to \pi_q(X,y_0)
\end{equation*}
is bijective, and the maps
\begin{equation*}
\pi_0(A) \to \pi_0(X_\lambda)
\quad\mbox{and}\quad
\pi_0(A) \to \pi_0(X_{\lambda_i}\setminus \Sigma_i)
\end{equation*}
are surjective (i.e., the pairs $(X_\lambda,A)$ and $(X_{\lambda_i}\setminus \Sigma_i,A)$ are $0$-connected).
Then, for any base point $x_0\in A$, there is an isomorphism
\begin{equation*}
\pi_1(X_\lambda,x_0)\bigg/ \overline{\bigcup_{1\leq i\leq N} 
\mbox{\emph{im}}(\mbox{\emph{Var}}^{\, \mbox{\tiny \emph{rel}}}_{h_i})} 
\overset{\sim}{\longrightarrow} \pi_1(X,x_0).
\end{equation*}
\end{theorem}

As in Theorems \ref{Eyral-1}, \ref{Eyral-2} and \ref{fmt}, all maps involved in Theorem \ref{mt} are induced by inclusions. Theorem \ref{mt} is a new generalization of the Zariski-van Kampen theorem to singular varieties.
It also generalizes Theorem \ref{Eyral-1}.

The rest of the paper is entirely devoted to the proof of Theorem \ref{mt}.

\section{Proof of Theorem \ref{mt}}\label{proof-mt}

Let $x_0$ be a base point in $A$. 
As observed (under weaker assumptions) at the beginning of Section \ref{sect-smr}, the bijectivity of $\pi_0(X_{\lambda}, x_0) \rightarrow \pi_0(X, x_0)$ and the surjectivity of $\pi_1(X_{\lambda}, x_0) \rightarrow \pi_1(X, x_0)$ follow from \cite[Th\'eor\`eme 2.5]{E2}. To prove Theorem \ref{mt}, it remains to show that the kernel of $\pi_1(X_{\lambda}, x_0) \rightarrow \pi_1(X, x_0)$ is equal to the normal subgroup of $\pi_1(X_{\lambda}, x_0)$ generated by the union of the images of the relative monodromy variation operators $\mbox{Var}^{\, \mbox{\tiny rel}}_{h_1},\ldots, \mbox{Var}^{\, \mbox{\tiny rel}}_{h_N}$. To prove this assertion, as in \cite{E1}, we first observe that it suffices to consider the special case $\Sigma \subseteq Z$. Indeed, assume that the assertion holds true in this case. Then,  proceeding as in \cite[\S 11]{C4} and \cite[\S 9.2]{E2}, we consider the proper closed algebraic subset $Z':=Z \cup \Sigma$ of $Y$ and we look at the new Whitney stratification $\Xi'$ of $Y$ the strata of which consist of the points of $\Sigma$ together with the traces on $Y\setminus \Sigma$ of the strata of $\Xi$. Clearly, $Z'$ is a union of such new strata. The axis $\Pi_0$ and the generic members of the pencil $\Pi$ transversely meet all the strata of $\Xi'$, while for each $1\leq i\leq N$, the hyperplane $L_i$ is transverse to all these strata except to those consisting of the points of $\Sigma_i$. Thus, if we consider $Z'$ and $\Xi'$ instead of $Z$ and $\Xi$, then we are in the situation of Theorem \ref{mt}, taking the same pencil $\Pi$, the same special hyperplanes $L_i$ ($1\leq i\leq N$), and hence the same sets $\Sigma_i$.
As $\Pi_0$, $L$ and $L_i\setminus \Sigma_i$ do not meet $\Sigma$, the assumptions of Theorem \ref{mt} imply the same assumptions with the set $Z'$ instead of $Z$, and the relative variation operators $\mbox{Var}^{\, \mbox{\tiny rel}}_{h_1},\ldots, \mbox{Var}^{\, \mbox{\tiny rel}}_{h_N}$ remain unchanged. As $\Sigma\subseteq Z'$ and since we have assumed that the assertion is true in this case, it follows that the kernel of the map $\pi_1(X_{\lambda}, x_0) \rightarrow \pi_1(Y\setminus Z', x_0)$ is given by the normal subgroup generated by the images of the relative variation operators $\mbox{Var}^{\, \mbox{\tiny rel}}_{h_i}$ for all $1\leq i\leq N$. 
Now, as $Y \setminus Z' = X \setminus \Sigma$, the general case follows from the  special case using the bijectivity of the map
$\pi_1(X \setminus \Sigma, x_0)\rightarrow \pi_1(X, x_0)$.

We must now prove that the kernel of the map $\pi_1(X_{\lambda}, x_0) \rightarrow \pi_1(X, x_0)$ is actually equal to the normal subgroup of $\pi_1(X_{\lambda}, x_0)$ generated by the images of the relative variation operators $\mbox{Var}^{\, \mbox{\tiny rel}}_{h_1},\ldots, \mbox{Var}^{\, \mbox{\tiny rel}}_{h_N}$ in the special case where $\Sigma \subseteq Z$. This covers the rest of the paper. The proof follows the same pattern as that of the classical Zariski-van Kampen theorem \cite{C2,C3,vK,Z} and its first singular version \cite{E1}. However, as we do not assume here that the map $\pi_0(A)\rightarrow \pi_0(X_\lambda)$ is bijective, it requires essential new arguments which lead us, in particular, to the relative monodromy variation. We shall also often refer to \cite{E2} for important results on the topology of singular spaces used in the proof.

From now on, we assume that $\Sigma \subseteq Z$.

\subsection{Blowing up and fibration outside the special hyperplanes}\label{sect-bu}

As in \cite{AF,C2,C3,C4,C1}, in order to translate crucial isotopies within the generic members of the pencil $\Pi$ in terms of locally trivial fibrations, it is convenient to blow up the ambient space $\mathbb{P}^n$ along the base locus $\Pi_0$ of $\Pi$. By definition, the \emph{blow up} of $\mathbb{P}^n$ along $\Pi_0$ is the $n$-dimensional compact analytic submanifold of $\mathbb{P}^n \times\mathbb{P}^1$ given by
\begin{equation*}
\widetilde{\mathbb{P}^n} := \{(x, \ell) \in \mathbb{P}^n \times\mathbb{P}^1 \mid x \in \Pi(\ell) \},
\end{equation*} 
where $\Pi(\ell) $ is the member of $\Pi$ with parameter $\ell$ (cf.~Section \ref{sect-varop}). The restrictions to $\widetilde{\mathbb{P}^n}$ of the projections of $\mathbb{P}^n \times\mathbb{P}^1$ give proper analytic morphisms
\begin{equation*}
f\colon\widetilde{\mathbb{P}^n}\rightarrow \mathbb{P}^n
\quad\mbox{and}\quad
p\colon\widetilde{\mathbb{P}^n}\rightarrow \mathbb{P}^1
\end{equation*} 
which are called the \emph{blowing up} morphism and the \emph{projection} morphism respectively.
\begin{notation}
For any subsets $E \subseteq \mathbb{P}^n$ and $\Lambda \subseteq \mathbb{P}^1$, we set
\begin{equation*}
\widetilde{E}:= f^{-1}(E)
\quad\mbox{and}\quad
\widetilde{E}_\Lambda := \widetilde{E} \cap p^{-1}(\Lambda).
\end{equation*} 
One must not confuse $\widetilde{E}_\Lambda$ with $\widetilde{E_\Lambda} = \widetilde{E}_\Lambda \cup \widetilde{(E \cap \Pi_0)}= \widetilde{E}_\Lambda \cup ((E \cap \Pi_0) \times \mathbb{P}^1)$. For instance, $\widetilde{X_\lambda}=\widetilde X_\lambda\cup (A\times\mathbb{P}^1)$, where, as above, we write $\widetilde X_\lambda$ instead of $\widetilde X_{\{\lambda\}}$.
\end{notation}

Taking a suitable stratification of $\widetilde{Y}$ and applying the first isotopy theorem of Thom-Mather \cite{M,T} shows that the restriction $p^*$ of $p$ to 
\begin{equation*}
\widetilde{X}^* := \widetilde{X} \bigg\backslash \bigg(\bigcup_{1\leq i\leq N}\widetilde{X}_{\lambda_i}\bigg)
\end{equation*} 
 is a locally trivial fibration over $\mathbb{P}^{1*}$ with fibre $\widetilde{X}_{\lambda}$ homeomorphic to $X_{\lambda}$. Moreover, this topological bundle has $A \times \mathbb{P}^{1*}$ as a trivial subbundle of it. For details we refer the reader to \cite[(11.1.5)]{C4}.

Clearly, the blowing up morphism $f$ induces an isomorphism
\begin{equation*}
\pi_1(\widetilde{X}_{\lambda}, (x_0, \lambda)) 
\overset{\sim}{\rightarrow} \pi_1(X_{\lambda}, x_0).
\end{equation*} 
As $(X_{\lambda}, A)$ is 0-connected and $\Sigma \subseteq Z$, it also induces an isomorphism
\begin{equation}\label{bumih}
\pi_1(\widetilde{X}, (x_0, \lambda)) 
\overset{\sim}{\rightarrow} \pi_1(X, x_0).
\end{equation} 
This second assertion is far from being obvious. It is proved in \cite[\S 8]{E2}. Roughly, the idea of the proof is as follows. The blowing up morphism $f$ induces an isomorphism $\widetilde{X} \setminus \widetilde{A} \overset{\sim}{\rightarrow} X \setminus A$. Then by applying the homotopy excision theorem of Blakers-Massey (see e.g.~\cite[Corollary 16.27]{G}) to a suitable excision in the mapping cylinder of the blowing up morphism, we can show that the map 
\begin{equation*}
\pi_q(\widetilde{X}, \widetilde{A}, (x_0, \lambda)) \rightarrow \pi_q(X, A, x_0)
\end{equation*}  
(induced by $f$) is bijective for $q=1$ and surjective for $q=2$. Then the bijectivity of (\ref{bumih}) can be obtained using properties of the projection morphism $p$. For a complete and detailed proof, we refer the reader to \cite[\S 8]{E2}.

In order to use the geometric setting described above, we include the natural map $\pi_1(X_{\lambda}, x_0) \rightarrow \pi_1(X, x_0)$ into the following commutative diagram, where the horizontal arrows are induced by inclusions and the vertical ones are induced by the blowing up morphism:
\begin{equation}\label{commdiag-1}
\begin{gathered}
\xymatrix{
\pi_1(\widetilde{X}_\lambda,(x_0, \lambda)) \ar[d]^{\wr}_{f_\#}
\ar@{}[dr]|{\circlearrowleft} \ar[r]^{\widetilde{\mbox{\tiny incl}}_\#}
& \pi_1(\widetilde{X},(x_0, \lambda)) \ar[d]^{\wr}  \\
\pi_1(X_\lambda, x_0)  \ar[r]^{\mbox{\tiny incl}_\#} & \pi_1(X, x_0)
}
\end{gathered}
\end{equation}
(Note that, since $\mbox{incl}_\#$ is surjective, so is $\widetilde{\mbox{incl}}_\#$.) Clearly, 
\begin{equation*}
\ker(\mbox{incl}_\#)=f_{\#} (\ker(\widetilde{\mbox{incl}}_\#)). 
\end{equation*}
Thus in order to prove the theorem it suffices to compute $\ker(\widetilde{\mbox{incl}}_\#)$. For that purpose, it is convenient to write $\widetilde{\mbox{incl}}_\#$ as the composite of the following maps (both induced by inclusions):
\begin{equation*}
\pi_1(\widetilde{X}_\lambda,(x_0, \lambda)) \rightarrow \pi_1(\widetilde{X}^*,(x_0, \lambda)) \rightarrow \pi_1(\widetilde{X},(x_0, \lambda)).
\end{equation*}
The plan for the rest of the proof is as follows. In \S \ref{fgfb}, we study the relationship between the groups $\pi_1(\widetilde{X}_\lambda,(x_0, \lambda))$ and $\pi_1(\widetilde{X}^*,(x_0, \lambda))$, and in \S \ref{fgtbus} we compare the groups $\pi_1(\widetilde{X}^*,(x_0, \lambda))$ and $\pi_1(\widetilde{X},(x_0, \lambda))$. To understand the relations between these groups, we introduce in \S \ref{rmvopbus} the relative monodromy variation operator on the blown up space.

\subsection{Relative monodromy variation operator on the blown up space}\label{rmvopbus}

By \cite[Lemma 4.2]{C1}, for each $1\leq i\leq N$, there is an isotopy
\begin{equation}\label{isotopyHit}
\widetilde H_i\colon \widetilde X_\lambda \times I
\to \widetilde X_{\partial K_i},\quad (x,\tau)\mapsto \widetilde H_i(x,\tau),
\end{equation}
satisfying the following properties:
\begin{enumerate}
\item
$\widetilde H_i((x,\lambda),0)=(x,\lambda)$ for any $(x,\lambda)\in\widetilde X_\lambda$;
\item 
$\widetilde H_i((x,\lambda),\tau)\in \widetilde X_{\omega_i(\tau)}$ for any $(x,\lambda)\in\widetilde X_\lambda$ and any $\tau\in I$;
\item
for each $\tau\in I$, the map $\widetilde X_\lambda \to \widetilde X_{\omega_i(\tau)}$, defined by $(x,\lambda)\mapsto \widetilde H_i((x,\lambda),\tau)$, is a homeomorphism;
\item 
$\widetilde H_i((x,\lambda),\tau)=(x,\omega_i(\tau))$ for any $(x,\lambda)\in \widetilde A_\lambda=A\times \{\lambda\}$ and any $\tau\in I$.
\end{enumerate}

\begin{remark}[\mbox{cf.~\cite[Lemma 4.2]{C1}}]\label{rk-relHiHit}
Any such isotopy $\widetilde{H}_i$ induces an isotopy $H_i$ as in (\ref{isotopyHi}) if we put $H_i(x,\tau):=f(\widetilde{H}_i((x,\lambda),\tau))$ for any $(x,\tau) \in X_\lambda \times I$. Conversely, any isotopy $H_i$ given by (\ref{isotopyHi}) can be obtained using the above formula from a unique isotopy $\widetilde{H}_i$ (as in (\ref{isotopyHit})) defined by $\widetilde{H}_i((x,\lambda),\tau):=(H_i(f(x,\lambda),\tau), \omega_i(\tau))=(H_i(x,\tau),\omega_i(\tau))$ for any $((x,\lambda),\tau) \in \widetilde{X}_\lambda \times I$.
\end{remark}

Clearly, the terminal homeomorphism $\widetilde h_i\colon \widetilde X_\lambda\to \widetilde X_\lambda$ of the above isotopy, which is defined by
\begin{equation*}
(x,\lambda)\mapsto \widetilde h_i(x,\lambda):=\widetilde H_i((x,\lambda),1),
\end{equation*}
leaves $\widetilde A_\lambda=A\times \{\lambda\}$ pointwise fixed.

\begin{definition}\label{def-mbus}
The map $\widetilde{h}_i$ is called a monodromy of $\widetilde{X}_\lambda$ relative to $\widetilde{A}_\lambda$ above $\omega_i$. 
(A similar observation to that made in Remark \ref{dohcoi} applies for $\widetilde{h}_i$ as well.)
\end{definition}

Now pick a relative loop $\alpha\in F^1(\widetilde{X}_\lambda, \widetilde A_\lambda, (x_0,\lambda))$, and consider the map $\alpha^{-1}\cdot \widetilde{h}_i\circ\alpha$ defined on~$I$ by
\begin{align*}
\alpha^{-1}\cdot \widetilde{h}_i\circ\alpha(t):= 
\left\{
\begin{aligned}
& \alpha^{-1}(2t)&\mbox{ for }\quad  0\leq t\leq 1/2,\\
& \widetilde{h}_i\circ\alpha(2t-1) &\mbox{ for } \quad 1/2\leq t\leq 1.
\end{aligned}
\right.
\end{align*}
As $\widetilde{h}_i$ is the identity on $\widetilde{A}_\lambda$, this map is well defined and it belongs to $F^1(\widetilde{X}_\lambda, (x_0, \lambda))$, that is, $\alpha^{-1}\cdot \widetilde{h}_i\circ\alpha$ is an absolute loop.
Then a similar argument to that given in \cite[Lemma 4.1]{CE} shows that the correspondence
\begin{align*}
\widetilde{\mbox{Var}}^{\mbox{\tiny rel}}_{\, \widetilde{h}_i}\colon \pi_1(\widetilde{X}_\lambda,\widetilde{A}_\lambda,(x_0, \lambda)) & \to \pi_1(\widetilde{X}_\lambda,(x_0, \lambda))\\ 
[\alpha]_{\widetilde{X}_\lambda,\widetilde{A}_\lambda,(x_0, \lambda)} & \mapsto 
[\alpha^{-1}\cdot \widetilde{h}_i\circ\alpha]_{\widetilde{X}_\lambda,(x_0, \lambda)}
\end{align*}
is well defined and only depends on the homotopy class $[\omega_i]\in \pi_1(\mathbb{P}^{1*},\lambda)$. 

Operators $\mbox{Var}^{\, \mbox{\tiny rel}}_{h_i}$ and $\widetilde{\mbox{Var}}^{\mbox{\tiny rel}}_{\, \widetilde{h}_i}$  are related to each other through (the homomorphism induced in homotopy by) the blowing up morphism. This is stated in the next lemma.

\begin{lemma}\label{commdiag-2}
The following diagram, in which the vertical maps are induced by the blowing up morphism $f$, commutes:
\begin{equation*}
\xymatrix{
\pi_1(\widetilde{X}_\lambda, \widetilde{A}_\lambda, (x_0, \lambda)) \ar[d]^{\wr}
\ar@{}[dr]|{\circlearrowleft} \ar[r]^{\widetilde{\emph{\mbox{Var}}}^{\emph{\mbox{\tiny rel}}}_{\, \widetilde{h}_i} } 
& \pi_1(\widetilde{X}_\lambda, (x_0, \lambda)) \ar[d]^{\wr} \\
\pi_1(X_\lambda, A, x_0)  \ar[r]^{\emph{\mbox{Var}}^{\, \emph{\mbox{\tiny rel}}}_{h_i}}& \pi_1(X_\lambda, x_0)
}
\end{equation*}
\end{lemma}
This lemma immediately follows from Remarks \ref{dohcoi} and \ref{rk-relHiHit}.

\subsection{The fundamental group $\pi_1(\widetilde X^*,(x_0,\lambda))$}\label{fgfb}

The main result of this section is Proposition \ref{fgxts}. This proposition is already proved in \cite[Lemma 7.3.3]{E1}. It is a singular version of \cite[Proposition (4.1.1)]{C2} (and \cite[Lemme (2.4)]{C3}). For convenience of the reader, we briefly recall the idea.

The exact homotopy sequence of the locally trivial fibration 
\begin{equation*}
p^* \colon \widetilde{X}^* \rightarrow \mathbb{P}^{1*}
\end{equation*} 
(induced by the projection morphism $p$) is written as follows:
\begin{equation*}
\cdots \rightarrow \underbrace{\pi_2(\mathbb{P}^{1*},\lambda)}_{\mbox{\tiny trivial group $\{e\}$}} \rightarrow \pi_1(\widetilde{X}_\lambda,(x_0, \lambda)) \overset{i_\#}{\rightarrow} \pi_1(\widetilde{X}^*,(x_0, \lambda)) \overset{p^*_\#}{\rightarrow} \pi_1(\mathbb{P}^{1*}, \lambda) \rightarrow\cdots,
\end{equation*} 	
where $i_\#$ is induced by inclusion.
As $p$ induces an isomorphism
\begin{equation}\label{isobase}
\pi_1(\{x_0\} \times \mathbb{P}^{1*}, (x_0, \lambda)) \overset{\sim}{\rightarrow} \pi_1(\mathbb{P}^{1*}, \lambda), 
\end{equation} 
we have the short exact sequence
\begin{equation}\label{exact1}
\{e\} \rightarrow \pi_1(\widetilde{X}_\lambda, (x_0, \lambda)) \overset{i_{\#}}{\longrightarrow} \pi_1(\widetilde{X}^*, (x_0, \lambda)) \overset{\rho}{\longrightarrow} \pi_1(\{x_0\} \times \mathbb{P}^{1*}, (x_0, \lambda)),
\end{equation}
where $\rho$ is the composite of $p^*_{\#}$ with the inverse of isomorphism (\ref{isobase}).

As $i_{\#}$ is injective, we can identify $\pi_1(\widetilde{X}_\lambda, (x_0, \lambda))$ with its image by $i_{\#}$. In other words, for any loop $\alpha\in F^1(\widetilde{X}_\lambda,(x_0, \lambda))$, we identify the homotopy classes
\begin{equation}\label{identhc1}
[\alpha]_{\widetilde{X}_\lambda, (x_0, \lambda)} 
\quad\mbox{and}\quad
[\alpha]_{\widetilde{X}^*, (x_0, \lambda)}.
\end{equation}
With such an identification, $\pi_1(\widetilde{X}_\lambda, (x_0, \lambda))$ can be viewed as a normal subgroup of $\pi_1(\widetilde{X}^*, (x_0, \lambda))$. 

Now consider the natural map 
\begin{equation*}
j_{\#} \colon \pi_1(\{x_0\} \times \mathbb{P}^{1*}, (x_0, \lambda))\rightarrow \pi_1(\widetilde{X}^*, (x_0, \lambda)).
\end{equation*}
We easily show that $j_{\#}$ is a \emph{section} of $\rho$, that is, the composite $\rho \circ j_{\#}$ is the identity. Moreover, as $j_{\#}$ is injective, we can identify
$\pi_1(\{x_0\} \times \mathbb{P}^{1*}, (x_0, \lambda))$ with its image by $j_{\#}$.
That is, for any loop $\alpha\in F^1(\{x_0\} \times \mathbb{P}^{1*}, (x_0, \lambda))$, we identify the homotopy classes
\begin{equation}\label{identhc2}
[\alpha]_{\{x_0\} \times \mathbb{P}^{1*}, (x_0, \lambda)} 
\quad\mbox{and}\quad
[\alpha]_{\widetilde{X}^*, (x_0, \lambda)}.
\end{equation}
Combined with the exactness of (\ref{exact1}), the existence of such a section $j_{\#}$ of $\rho$ implies that $\pi_1(\widetilde{X}^*, (x_0, \lambda))$ is the \emph{internal semi-direct product} of its subgroups $\pi_1(\widetilde{X}_\lambda, (x_0, \lambda))$ and $\pi_1(\{x_0\} \times \mathbb{P}^{1*}, (x_0, \lambda))$.
Then, by \cite[Proposition 10.1 and Corollary 10.1]{J}, we obtain the following presentation for the fundamental group $\pi_1(\widetilde{X}^*, (x_0, \lambda))$.

\begin{proposition}[\mbox{cf.~\cite[Lemma 7.3.3]{E1}}]\label{fgxts}
Fix a presentation of $\pi_1(\widetilde{X}_\lambda, (x_0, \lambda))$ as in \cite[Proposition 4.1]{J}. 
Then the fundamental group $\pi_1(\widetilde{X}^*, (x_0, \lambda))$ is presented by the generators of $\pi_1(\widetilde{X}_\lambda, (x_0, \lambda))$, the generators 
\begin{equation*}
[(x_0, \omega_1)], \dots, [(x_0, \omega_N)]
\end{equation*}
of $\pi_1(\{x_0\} \times \mathbb{P}^{1*}, (x_0, \lambda))$, 
and by the relations of $\pi_1(\widetilde{X}_\lambda, (x_0, \lambda))$ together with the following additional relations:
\begin{enumerate}
\item[(i)]
$[(x_0, \omega_1)] \cdots [(x_0, \omega_N)] = e$;
\item[(ii)]
$[\alpha]\cdot[(x_0, \omega_i)] = [(x_0, \omega_i)]\cdot\widetilde{h}_{i\#}([\alpha])$ for any $1 \leq i \leq N$ and any $[\alpha] \in \pi_1(\widetilde{X}_\lambda, (x_0, \lambda))$;
\end{enumerate}
where $\widetilde{h}_{i\#}\colon \pi_1(\widetilde{X}_\lambda, (x_0, \lambda))\overset{\sim}{\rightarrow} \pi_1(\widetilde{X}_\lambda, (x_0, \lambda))$ is the automorphism induced by the monodromy $\widetilde{h}_{i}$ (cf.~Definition \ref{def-mbus}) and $N$ is the number of special hyperplanes.
\end{proposition}

Here, $(x_0, \omega_i)$ denotes the loop $I \rightarrow \{x_0\} \times \mathbb{P}^{1*}$ defined by $t\mapsto (x_0, \omega_i)(t) := (x_0, \omega_i(t))$.

\subsection{The fundamental group $\pi_1(\widetilde X,(x_0,\lambda))$}\label{fgtbus}
In this section, we prove the following proposition, which extends Proposition (4.2.1) of \cite{C2} and Lemma 7.4.1 of \cite{E1}. This proposition is the main point in the proof of Theorem \ref{mt}.

\begin{proposition}\label{fgxt}
Choose a presentation of $\pi_1(\widetilde X_\lambda,(x_0,\lambda))$ as in Proposition \ref{fgxts}. Then the fundamental group $\pi_1(\widetilde X,(x_0,\lambda))$ is presented by the generators and the relations of $\pi_1(\widetilde X_\lambda,(x_0,\lambda))$ together with the additional relations
\begin{equation*}
\widetilde{\mbox{\emph{Var}}}^{\mbox{\tiny \emph{rel}}}_{\, \widetilde h_i}([\alpha])=e
\end{equation*}
for any $1\leq i\leq N$ and any $[\alpha]\in \pi_1(\widetilde X_\lambda,\widetilde A_\lambda,(x_0,\lambda))$.
In other words, the kernel of the natural epimorphism 
\begin{equation*}
\pi_1(\widetilde X_\lambda,(x_0,\lambda)) \to \pi_1(\widetilde X,(x_0,\lambda))
\end{equation*} 
coincides with the normal subgroup of $\pi_1(\widetilde X_\lambda,(x_0,\lambda))$ generated by the images of operators $\widetilde{\mbox{\emph{Var}}}^{\mbox{\tiny \emph{rel}}}_{\, \widetilde h_1},\ldots,\widetilde{\mbox{\emph{Var}}}^{\mbox{\tiny \emph{rel}}}_{\, \widetilde h_N}$ defined in Section \ref{rmvopbus}.
\end{proposition}

We divide the proof of this proposition into two key observations.
For any $1\leq i\leq N$ and any \emph{relative} loop $\alpha\in F^1(\widetilde X_\lambda,\widetilde A_\lambda,(x_0,\lambda))$, the composition of~paths
\begin{equation}\label{compfond-2}
\alpha^{-1} \cdot (f\circ\alpha(0),\omega_i) \cdot \alpha 
\end{equation}
defines an element of $F^1(\widetilde X^*,(x_0,\lambda))$ which is null-homotopic in $(\widetilde X,(x_0,\lambda))$. (We recall that $f$ is the blowing up morphism and that $(f\circ\alpha(0),\omega_i)$ denotes the loop $t\in I\mapsto (f\circ\alpha(0),\omega_i(t))\in \widetilde X^*$.) In other words, the normal subgroup $G$ of $\pi_1(\widetilde X^*,(x_0,\lambda))$ generated by the (homotopy classes of) loops of the form (\ref{compfond-2}) is contained in the kernel of the natural~map 
\begin{equation}\label{ixtsdxt}
\pi_1(\widetilde X^*,(x_0,\lambda)) \rightarrow\pi_1(\widetilde X,(x_0,\lambda)).
\end{equation}
The first crucial observation says that $G$ is actually equal to the kernel of this map.

\begin{lemma}\label{lemma-fget}
The normal subgroup $G$ coincides with the kernel of the map (\ref{ixtsdxt}).
\end{lemma}

In order to state the second key lemma, we consider the normal subgroup $G'$ of $\pi_1(\widetilde X^*,(x_0,\lambda))$ generated by the loops 
\begin{equation*}
t\in I\mapsto (x_0,\omega_i(t))\in A\times\mathbb{P}^{1*}
\quad\mbox{and}\quad
t\in I\mapsto (\alpha^{-1}\cdot\widetilde h_i\circ\alpha) (t) \in \widetilde X_\lambda
\end{equation*} 
for any $1\leq i\leq N$ and any $\alpha\in F^1(\widetilde X_\lambda,\widetilde A_\lambda,(x_0,\lambda))$. 

\begin{lemma}\label{lemma-fget-2}
The normal subgroups $G$ and $G'$ coincide.
\end{lemma}

Combined with \cite[Corollaire~5.3]{E2}, these lemmas imply Proposition \ref{fgxt}. Indeed, by \cite[Corollaire~5.3]{E2}, $\pi_1(\widetilde X,\widetilde X^*,(x_0,\lambda))=\{e\}$. Therefore, by the exact homotopy sequence of the pointed pair $(\widetilde X,\widetilde X^*)$ (with base point $(x_0,\lambda)$), the natural map (\ref{ixtsdxt}) is surjective. Now, by Lemmas \ref{lemma-fget} and \ref{lemma-fget-2}, its kernel is $G'$. Therefore, there is a natural isomorphism 
\begin{equation*}
\pi_1(\widetilde X^*,(x_0,\lambda)) \big/ G' \overset{\sim}{\rightarrow}\pi_1(\widetilde X,(x_0,\lambda)),
\end{equation*}
and a presentation of $\pi_1(\widetilde X,(x_0,\lambda))$ is obtained from the presentation of $\pi_1(\widetilde X^*,(x_0,\lambda))$ given in Proposition \ref{fgxts} only by adding the relations
\begin{equation*}
[(x_0,\omega_i)]_{\widetilde{X}^*,(x_0,\lambda)}=e
\quad\mbox{and}\quad 
\widetilde{\mbox{Var}}^{\mbox{\tiny rel}}_{\, \widetilde h_i}([\alpha]):=[\alpha^{-1}\cdot\widetilde h_i\circ\alpha]_{\widetilde{X}^*,(x_0,\lambda)}=e
\end{equation*}
for any $1\leq i\leq N$ and any $[\alpha]\in \pi_1(\widetilde X_\lambda,\widetilde A_\lambda,(x_0,\lambda))$.
(Remind the identifications (\ref{identhc1}) and (\ref{identhc2}).) Proposition \ref{fgxt} follows.

To complete the proof of the proposition, we must now prove the key lemmas \ref{lemma-fget} and \ref{lemma-fget-2}. Let us start with the proof of Lemma \ref{lemma-fget}.

\begin{proof}[Proof of Lemma \ref{lemma-fget}]
It only remains to prove that the kernel of the natural map (\ref{ixtsdxt}) is contained in $G$.  
Precisely, we must show that if $\alpha$ is any element of $F^1(\widetilde X^*,(x_0,\lambda))$ which is null-homotopic in $(\widetilde X,(x_0,\lambda))$ (i.e., $[\alpha]_{\widetilde X,(x_0,\lambda)}=e$), then 
\begin{equation*}
[\alpha]_{\widetilde X^*,(x_0,\lambda)}\in G.
\end{equation*} 

Let $\alpha\in F^1(\widetilde X^*,(x_0,\lambda))$ with $[\alpha]_{\widetilde X,(x_0,\lambda)}=e$. By the exact homotopy sequence of the pointed pair $(\widetilde X,\widetilde X^*)$, we have
\begin{equation*}
[\alpha]_{\widetilde X^*,(x_0,\lambda)} =
[\partial\beta]_{\widetilde X^*,(x_0,\lambda)},
\end{equation*}
where $\beta$ is a relative homotopy $2$-cell of $(\widetilde X,\widetilde X^*,(x_0,\lambda))$ and $\partial\beta$ is its boundary.\footnote{By a relative homotopy $2$-cell of $(\widetilde X,\widetilde X^*,(x_0,\lambda))$, we mean a map from the square $I^2$ to $\widetilde X$ with the face $\{(t_1, t_2)\in I^2 \mid t_2 = 0\}$ sent into $\widetilde X^*$ and all other faces sent to the base point $(x_0,\lambda)$. As usual, the boundary $\partial\beta$ of a relative homotopy $2$-cell $\beta$ is the absolute loop defined by the formula $\partial\beta(t):=\beta(t,0)$ for any $t\in I$. See \cite[\S 15]{St}.}
By \cite[Proposition 5.2]{E2}, we may assume that the set
\begin{equation*}
\beta^{-1}\biggl(\bigcup_{1\leq i\leq N} \widetilde X_{\lambda_i}\biggr)
\end{equation*}
is either empty or consists of finitely many points $P_1,\ldots, P_{k_0}$. Clearly, if this set is empty, then $[\alpha]_{\widetilde X^*,(x_0,\lambda)}=e\in G$ and we are done. Now, if 
\begin{equation*}
\emptyset\not=\beta^{-1}\biggl(\bigcup_{1\leq i\leq N} \widetilde X_{\lambda_i}\biggr)=\{P_k\in I^2\, ;\, 1\leq k\leq k_0\},
\end{equation*}
then, for each $k$, we pick a small closed disc $\Delta_k$ centred at $P_k$ such that $\Delta_k\cap\Delta_{k'}=\emptyset$ whenever $k\not=k'$, and we consider a loop $\gamma_k\colon I\to\partial \Delta_k$ which runs once counterclockwise in the boundary $\partial \Delta_k$ of $\Delta_k$.

\begin{sublemma}\label{lemma-freehomotopy}
Fix an index $k$ ($1\leq k\leq k_0$), and suppose that the corresponding point $P_k$ belongs to the subset $\beta^{-1}(\widetilde X_{\lambda_{i(k)}})$ for some $1\leq i(k)\leq N$ depending on $k$. 
If $\Delta_k$ is small enough, then the loop $\beta\circ\gamma_k$ is freely homotopic in $\widetilde X^*$ to a loop $\gamma_k'$ based at $(x_k,\ell_{i(k)})$ and the image of which is contained in $\{x_k\}\times\partial D_{i(k)}\subseteq \widetilde A^*:=A\times\mathbb{P}^{1*}$, where $x_k$ is a point of $A$ and where $\ell_{i(k)}$ and $\partial D_{i(k)}$ are as in Section \ref{sect-varop}. 
\end{sublemma}

Here, by a \emph{free} homotopy between $\beta\circ\gamma_k$ and $\gamma_k'$, we mean a homotopy 
\begin{equation*}
\varphi\colon I\times I\to\widetilde X^*,\quad 
(t,\tau)\mapsto\varphi(t,\tau),
\end{equation*}
from the map $\beta\circ\gamma_k\colon I\to \widetilde X^*$ to the map $\gamma_k'\colon I\to \widetilde X^*$ such that for each parameter $\tau\in I$, the map 
\begin{equation*}
t\in I\mapsto \varphi_\tau(t):=\varphi(t,\tau)\in \widetilde X^*
\end{equation*}
 is a loop (i.e., $\varphi_\tau(0)=\varphi_\tau(1)$).

\begin{remark}
Sublemma \ref{lemma-freehomotopy} corresponds to Lemma 7.4.2 of \cite{E1}. However, unlike the latter, since we do not assume that the map $\pi_0(A)\rightarrow \pi_0(X_\lambda)$ is injective, the point $x_k$ may be different from the base point $x_0$. This is a crucial difference with \cite{E1} and the reason which leads us to the \emph{relative} variation.
\end{remark}

\begin{proof}[Proof of Sublemma \ref{lemma-freehomotopy}]
To simplify, hereafter we write ``$i$'' instead of ``$i(k)$''.
For each $s\in\Sigma_i$, pick a small closed ball $\bar B_\varepsilon(s)\subseteq\widetilde{\mathbb{P}^n}$ with centre $s$ and radius $\varepsilon>0$ such that the following four conditions hold true:
\begin{enumerate}
\item
$\bar B_\varepsilon(s)\cap \bar B_\varepsilon(s')=\emptyset$ whenever $s\not=s'$;
\item
$\bar B_\varepsilon(s)\cap \widetilde{\Pi_0}=\emptyset$;
\item
$\bar B_\varepsilon(s)\cap p^{-1}(\lambda_j)=\emptyset$, where $\lambda_j$ is the parameter of a special hyperplane $L_j$ of $\Pi$ such that $L_j\not=L_i$;
\item
$\bar B_\varepsilon(s)\cap \mbox{im}(\beta)=\emptyset$ (this is possible as $\mbox{im}(\beta)$ is compact and does not intersect the set $\Sigma_i$, which is contained in $Z$).
\end{enumerate}

Take an open disc $U_i$ with centre $\lambda_i$ and radius $r_i$ such that $U_i\subseteq D_i$, and set
\begin{equation*}
E_i:=\bigcup_{s\in \Sigma_i} \Bigl(B_\varepsilon(s)\cap 
\widetilde X_{U_i}\Bigr),
\end{equation*}
where $B_\varepsilon(s)$ is the open ball with centre $s$ and radius $\varepsilon$. 

\begin{claim}\label{claim-exitence-path}
There exists a point $x_k\in A$ together with a path 
$\mu_k\colon I\to \widetilde X_{\lambda_i}\setminus E_i$ 
such that $\mu_k(0)=\beta(P_k)$ and $\mu_k(1)=(x_k,\lambda_i)$.
\end{claim}

\begin{proof}
Since $\Sigma_i\subseteq Z$, by applying the local conic structure lemma of D. Burghelea and A. Verona (cf.~\cite[Lemma 3.2]{BV}) to the set $\widetilde Y_{\lambda_i}$ (equipped with an appropriate Whitney stratification), it follows that if $\varepsilon$ is small enough, then the set $\widetilde X_{\lambda_i}\setminus E_i$ is a strong deformation retract of $\widetilde X_{\lambda_i}=\widetilde X_{\lambda_i}\setminus (\Sigma_i\times\{\lambda_i\})$. Combined with the surjectivity of 
\begin{equation*}
\pi_0(\widetilde A_{\lambda_i})\rightarrow
\pi_0(\widetilde X_{\lambda_i})
\end{equation*}
(which follows from that of $\pi_0(A)\rightarrow
\pi_0(X_{\lambda_i})=\pi_0(X_{\lambda_i}\setminus\Sigma_i)$), this implies the surjectivity of
\begin{equation*}
\pi_0(\widetilde A_{\lambda_i})\rightarrow
\pi_0(\widetilde X_{\lambda_i}\setminus E_i),
\end{equation*}
where all the maps are induced by inclusions.
The claim follows immediately.
\end{proof}

By \cite[Proposition 5.4]{E2} and  \cite[Remark 7.4.5]{E1} (applied with $x_k$ instead of $x_0$), if $\varepsilon$ is small enough and if $r_i\ll\varepsilon$, then there is a trivialization
\begin{equation*}
\psi_i\colon \widetilde X_{U_i}\setminus E_i \overset{\sim}{\rightarrow} U_i\times F_i
\end{equation*}
of the restriction of the projection morphism $p$ to the set $\widetilde X_{U_i}\setminus E_i$ such that:
\begin{enumerate}
\item
$\psi_i(\widetilde X_{(U_i\setminus\{\lambda_i\})}\setminus E_i)=(U_i\setminus\{\lambda_i\})\times F_i$;
\item
$\psi_i(A\times U_i)=U_i\times F_i'$;
\item
$\psi_i(\{x_k\}\times U_i)=U_i\times\{p_2\circ\psi_i(x_k,\lambda_i)\}$;
\end{enumerate}
where the pair $(F_i,F_i')$ is homeomorphic to the pair 
\begin{equation*}
(\widetilde X_{\lambda_i}\setminus E_i,\widetilde A_{\lambda_i}\setminus E_i)=(\widetilde X_{\lambda_i}\setminus E_i,\widetilde A_{\lambda_i}),
\end{equation*} 
and where $p_2$ is the second projection of $U_i\times F_i$. By Claim \ref{claim-exitence-path}, the image of the path 
\begin{equation*}
p_2\circ\psi_i\circ\mu_k
\end{equation*}
is contained in $F_i$, it starts at $p_2\circ\psi_i\circ\beta(P_k)$ and ends at $p_2\circ\psi_i(x_k,\lambda_i)$. 
Clearly, we may assume that the disc $\Delta_k$ is small enough so that
\begin{equation*}
\beta(\Delta_k\setminus\{P_k\}) \subseteq 
\widetilde X_{U_i\setminus\{\lambda_i\}}\setminus E_i.
\end{equation*} 
Then, as $P_k$ is a strong deformation retract of $\Delta_k$, the loop $\psi_i\circ\beta\circ\gamma_k$ is freely homotopic in the product $U_i\times F_i$ to the constant loop based at $\psi_i\circ\beta(P_k)$, and hence the loop $p_2\circ\psi_i\circ\beta\circ\gamma_k$ is freely homotopic in $F_i$ to the constant loop based at $p_2\circ\psi_i\circ\beta(P_k)$. It follows that the loop $\beta\circ\gamma_k$ is freely homotopic in $\widetilde X_{U_i\setminus\{\lambda_i\}}\setminus E_i\subseteq \widetilde X^*$ to the loop $\nu_k$ defined by
\begin{equation*}
t\in I\mapsto \nu_k(t):=\psi_i^{-1}(p_1\circ\psi_i\circ\beta\circ\gamma_k(t),p_2\circ\psi_i\circ\beta(P_k)),
\end{equation*}
where $p_1$ and $p_2$ are the first and second projections of $U_i\times F_i$ respectively. Now it is easy to see that this loop $\nu_k$ is freely homotopic in $\widetilde X_{U_i\setminus\{\lambda_i\}}\setminus E_i$ to a loop $\gamma_k'$ the image of which is contained in 
\begin{equation*}
\{x_k\}\times(U_i\setminus\{\lambda_i\})
\subseteq \{x_k\}\times (D_i\setminus \{\lambda_i\})\subseteq \widetilde A^*:=A\times\mathbb{P}^{1*}. 
\end{equation*}
For instance, the map 
\begin{equation*}
I\times I\to \widetilde X_{U_i\setminus\{\lambda_i\}}\setminus E_i
\end{equation*} 
defined by
\begin{equation*}
(t,\tau)\mapsto \psi_i^{-1}(p_1\circ\psi_i\circ\beta\circ\gamma_k(t),p_2\circ\psi_i\circ\mu_k(\tau))
\end{equation*}
is a free homotopy from $\nu_k$ to such a loop $\gamma_k'$. (The inclusion $\mbox{im}(\gamma_k')\subseteq \{x_k\}\times(U_i\setminus\{\lambda_i\})$ follows from  properties (1)--(3) of the trivialization $\psi_i$.) Moreover, as we are dealing here with free homotopies, by using the standard strong deformation retraction from $D_i\setminus\{\lambda_i\}$ to $\partial D_i$, we may always assume that $\mbox{im}(\gamma_k')$ is actually contained in $\{x_k\}\times\partial D_i$ and that $\gamma_k'$ starts (and ends) at $(x_k,\ell_{i})$ as desired.
 
This completes the proof of Sublemma \ref{lemma-freehomotopy}.
\end{proof}

Now we can complete the proof of Lemma \ref{lemma-fget}.
By Sublemma \ref{lemma-freehomotopy}, $\beta\circ\gamma_k$ is freely homotopic in $\widetilde X^*$ to a loop $\gamma_k'$ based at $(x_k,\ell_{i(k)})$ and such that $\mbox{im}(\gamma_k')\subseteq\{x_k\}\times\partial D_{i(k)}$. It immediately follows that
$\beta\circ\gamma_k$ is homotopic in $(\widetilde X^*,\beta\circ\gamma_k(0))$ to a loop of the form 
\begin{equation}\label{homozeta}
\zeta_k\gamma_k'\zeta_k^{-1},
\end{equation}
where $\zeta_k$ is a path in $\widetilde X^*$ such that $\zeta_k(0)=\beta\circ\gamma_k(0)$ and $\zeta_k(1)=\gamma_k'(0)=(x_k,\ell_{i(k)})$. 
Clearly, the fundamental group
\begin{equation*}
\pi_1(I^2\setminus\{P_1,\ldots,P_{k_0}\},O),
\end{equation*}
where $O$ is the origin in $I^2$, is generated by loops of the form
\begin{equation*}
\xi_k\gamma_k\xi_k^{-1} 
\quad\mbox{for}\quad
1\leq k\leq k_0,
\end{equation*}
where $\xi_k$ is a simple path from $O$ to $\gamma_k(0)\in\partial \Delta_k$ such that:
\begin{enumerate}
\item
$\mbox{im}(\xi_k)\cap\Delta_k=\{\gamma_k(0)\}$;
\item
$\mbox{im}(\xi_k)\cap \mbox{im}(\xi_{k'})=\emptyset$ whenever $k\not=k'$;
\item
$\mbox{im}(\xi_k)\cap\Delta_{k'}=\emptyset$ whenever $k\not=k'$.
\end{enumerate}
Taking a counterclockwise parametrization of the boundary of $I^2$ gives a loop based at $O$ and homotopic in $(I^2\setminus\{P_1,\ldots,P_{k_0}\},O)$ to the loop
\begin{equation*}
\prod_{1\leq k\leq k_0} \xi_{k}\gamma_{k}\xi_{k}^{-1} :=
\xi_1\gamma_1\xi_1^{-1} \cdots \xi_{k_0}\gamma_{k_0}\xi_{k_0}^{-1}
\end{equation*}
(by reordering if necessary). It follows that $\alpha$ is homotopic in $(\widetilde X^*,(x_0,\lambda))$ to the loop
\begin{equation*}
\prod_{1\leq k\leq k_0}  
\beta\circ(\xi_{k}\gamma_{k}\xi_{k}^{-1}),
\end{equation*}
and hence, by (\ref{homozeta}), to the loop
\begin{equation*}
\prod_{1\leq k\leq k_0}
(\beta\circ\xi_k)\cdot\zeta_k\gamma_{k}'\zeta_k^{-1}\cdot(\beta\circ\xi_k)^{-1}.
\end{equation*}

\begin{claim}\label{existence-rel-loop}
For each $k$, there exists a path $\sigma_k\colon I\to\widetilde X_{\lambda}$ such that $\sigma_k(0)=(x_k,\lambda)$ and $\sigma_k(1)=(x_0,\lambda)$.
\end{claim}

\begin{proof}
This claim is far from being obvious. It follows from the hyperplane section theorem for pencils \cite[Th\'eor\`eme 2.5]{E2}. More precisely,
let $\theta_k\colon I\to \{x_k\}\times \mathbb{P}^{1*}$ be the path defined by
\begin{equation*}
\theta_k(t):=(x_k,\rho_{i(k)}(t)),
\end{equation*}
where $\rho_{i(k)}$ is the tail of the loop $\omega_{i(k)}$---that is, $\rho_{i(k)}$ is the simple path in $\mathbb{P}^{1*}$ joining $\lambda$ to $\ell_{i(k)}$ as defined in Section \ref{sect-varop}. Thus, $\theta_k(0)=(x_k,\lambda)$ and $\theta_k(1)=(x_k,\ell_{i(k)})=\gamma_k'(0)$. Then
\begin{equation*}
\beta\circ\xi_k\cdot\zeta_k\cdot\theta_k^{-1}
\end{equation*}
is a path in $\widetilde X^*\subseteq\widetilde X$ starting at $\beta\circ\xi_k(0)=\alpha(0)=(x_0,\lambda)$ and ending at $\theta_k^{-1}(1)=(x_k,\lambda)$. It follows that $f(x_0,\lambda)=x_0\in A\subseteq X_\lambda$ can be joined to $f(x_k,\lambda)=x_k\in A$ by a path in $X$. Now, by \cite[Th\'eor\`eme 2.5]{E2}, the natural map $\pi_0(X_\lambda)\to \pi_0(X)$ is bijective. Therefore $x_0$ can be joined to $x_k$ in $X_\lambda$. The claim follows immediately.
\end{proof}

Clearly, the loop $\alpha$ is homotopic in $(\widetilde X^*,(x_0,\lambda))$ to the loop
\begin{equation*}
\prod_{1\leq k\leq k_0}((\beta\circ\xi_k)\zeta_k\theta_k^{-1}\sigma_k)\cdot 
(\sigma_k^{-1}\cdot\theta_k\gamma_k'\theta_k^{-1}\cdot \sigma_k) \cdot 
((\beta\circ\xi_k)\zeta_k\theta_k^{-1}\sigma_k)^{-1},
\end{equation*}
which is an element of the normal subgroup $G$ of $\pi_1(\widetilde X^*,(x_0,\lambda))$.
Indeed, the loop $\sigma_k^{-1}\cdot\theta_k\gamma_k'\theta_k^{-1}\cdot\sigma_k$ is homotopic in $(\widetilde X^*,(x_0,\lambda))$ to a (possibly negative) power of the loop 
\begin{equation*}
\sigma_k^{-1}\cdot (f\circ\sigma_k(0),\omega_{i(k)})\cdot\sigma_k.
\end{equation*}
(Note that $\sigma_k\in F^1(\widetilde X_\lambda,\widetilde A_\lambda, (x_0,\lambda))$, and hence the above loop is a generator of the normal subgroup $G$.)

This completes the proof of Lemma \ref{lemma-fget}.
\end{proof}

We now prove Lemma \ref{lemma-fget-2}.

\begin{proof}[Proof of Lemma \ref{lemma-fget-2}]
Let $\alpha \in F^1(\widetilde{X}_{\lambda}, 
\widetilde{A}_{\lambda},(x_0,\lambda))$ and let $1\leq i\leq N$. 

\begin{claim}\label{fccp}
The relative loops $\alpha$ and $(f\circ\alpha(0),\omega_i)\cdot \widetilde h_i\circ\alpha\cdot (f\circ\alpha(1),\omega_i)^{-1}$ are homotopic in $(\widetilde{X}^*, \widetilde{A}_{\lambda},(x_0,\lambda))$. Moreover, if 
\begin{equation*}
T\colon I\times I\to \widetilde{X}^*,\quad
(t,\tau)\mapsto T(t,\tau),
\end{equation*} 
denotes such a homotopy, then we can always choose it so that $T(0,\tau)=(f\circ\alpha(0),\lambda)$ for any parameter $\tau\in I$.
\end{claim}

\begin{proof}
It is similar to the proof of \cite[Lemma 7.3.3]{E1}. The only difference is that, in \cite{E1}, loops are \emph{absolute} whereas we are dealing here with \emph{relative} loops. Let $ \widetilde{H}_i$ be an isotopy underlying the monodromy $\widetilde h_i$ (cf.~Section \ref{rmvopbus}). Then the map $T\colon I\times I\to \widetilde{X}^*$ defined by
\begin{equation*}
(t,\tau) \mapsto
\left\{
\begin{aligned}
& (f\circ\alpha(0),\omega_i(3t\tau)) & \mbox{for} && 0 \leq t \leq 1/3\\
& \widetilde{H}_i(\alpha(3t-1), \tau) & \mbox{for} && 1/3 \leq t \leq 2/3\\
& (f\circ\alpha(1),\omega_i(3\tau(1-t))) & \mbox{for} && 2/3 \leq t \leq 1\\
\end{aligned}
\right.
\end{equation*}
is a homotopy in $(\widetilde{X}^*, \widetilde{A}_{\lambda},(x_0,\lambda))$ from the relative loop
\begin{equation*}
t \mapsto
\left\{
\begin{aligned}
& (f\circ\alpha(0),\lambda) & \mbox{for} && 0 \leq t \leq 1/3\\
& \alpha(3t-1)& \mbox{for} && 1/3 \leq t \leq 2/3\\
& (x_0,\lambda) & \mbox{for} && 2/3 \leq t \leq 1\\
\end{aligned}
\right.
\end{equation*}
(which is clearly homotopic to $\alpha$ in $(\widetilde{X}^*, 
\widetilde{A}_{\lambda},(x_0,\lambda))$) to the relative loop
\begin{equation*}
t \mapsto
\left\{
\begin{aligned}
& (f\circ\alpha(0),\omega_i(3t)) & \mbox{for} && 0 \leq t \leq 1/3\\
& \widetilde h_i\circ\alpha(3t-1)& \mbox{for} && 1/3 \leq t \leq 2/3\\
& (x_0,\omega_i(3(1-t))) & \mbox{for} && 2/3 \leq t \leq 1\\
\end{aligned}
\right.
\end{equation*}
(which is obviously homotopic to the relative loop  $(f\circ\alpha(0),\omega_i)\cdot \widetilde h_i\circ\alpha\cdot (f\circ\alpha(1),\omega_i)^{-1}$ in $(\widetilde{X}^*, \widetilde{A}_{\lambda},(x_0,\lambda))$). Clearly, the homotopy $T$ is such that $T(0,\tau)=(f\circ\alpha(0),\lambda)$ for any $\tau\in I$.
\end{proof}

\begin{claim}\label{sccp}
The absolute loop $\alpha^{-1}\cdot(f\circ\alpha(0),\omega_i)\cdot \widetilde h_i\circ\alpha\cdot (f\circ\alpha(1),\omega_i)^{-1}$ is null-homotopic in $(\widetilde{X}^*, (x_0,\lambda))$. In other words, the absolute loops
\begin{equation*}
\alpha^{-1}\cdot(f\circ\alpha(0),\omega_i)\cdot \widetilde h_i\circ\alpha
\quad\mbox{and}\quad
(f\circ\alpha(1),\omega_i)=(x_0,\omega_i)
\end{equation*}
define the same homotopy class in $\pi_1(\widetilde{X}^*, (x_0,\lambda))$.
\end{claim}

\begin{proof}
If $T$ is a homotopy as in Claim \ref{fccp}, then the map $I\times I\to \widetilde{X}^*$ defined by
\begin{equation*}
(t,\tau) \mapsto
\left\{
\begin{aligned}
& \alpha^{-1}(2t) & \mbox{for} && 0 \leq t \leq 1/2\\
& T(2t-1, \tau) & \mbox{for} && 1/2 \leq t \leq 1\\
\end{aligned}
\right.
\end{equation*}
is a homotopy in $(\widetilde{X}^*, (x_0,\lambda))$ from the loop
\begin{equation*}
t\mapsto
\left\{
\begin{aligned}
& \alpha^{-1}(2t) & \mbox{for} && 0 \leq t \leq 1/2\\
& (f\circ\alpha(0),\lambda) & \mbox{for} && 1/2 \leq t \leq 2/3\\
& \alpha(3(2t-1)-1)& \mbox{for} && 2/3 \leq t \leq 5/6\\
& (x_0,\lambda) & \mbox{for} && 5/6 \leq t \leq 1\\
\end{aligned}
\right.
\end{equation*}
(which is homotopic in $(\widetilde{X}_{\lambda}, (x_0,\lambda))$ to the constant loop $t\in I\mapsto (x_0,\lambda)\in\widetilde{X}_\lambda$) to the loop
\begin{equation*}
t\mapsto
\left\{
\begin{aligned}
& \alpha^{-1}(2t) & \mbox{for} && 0 \leq t \leq 1/2\\
& (f\circ\alpha(0),\omega_i(3(2t-1))) & \mbox{for} && 1/2 \leq t \leq 2/3\\
& \widetilde h_i\circ\alpha(3(2t-1)-1)& \mbox{for} && 2/3 \leq t \leq 5/6\\
& (x_0,\omega_i(3(1-(2t-1))) & \mbox{for} && 5/6 \leq t \leq 1\\
\end{aligned}
\right.
\end{equation*}
(which is obviously homotopic to the loop $\alpha^{-1}\cdot(f\circ\alpha(0),\omega_i)\cdot \widetilde h_i\circ\alpha\cdot (f\circ\alpha(1),\omega_i)^{-1}$ in $(\widetilde{X}^*, (x_0,\lambda))$).
\end{proof}

Clearly, for any $1\leq i\leq N$ and any $\alpha \in F^1(\widetilde{X}_{\lambda}, 
\widetilde{A}_{\lambda},(x_0,\lambda))$, the compositions
\begin{equation}\label{compfond}
\alpha^{-1} \cdot (f\circ\alpha(0),\omega_i)
\cdot \widetilde h_{i}\circ \alpha 
\quad\mbox{and}\quad
\alpha^{-1} \cdot \widetilde h_{i}\circ \alpha
\end{equation}
define elements of $F^1(\widetilde X^*,(x_0,\lambda))$ and $F^1(\widetilde X_\lambda,(x_0,\lambda))\subseteq F^1(\widetilde X^*,(x_0,\lambda))$ respectively.
Consider the normal subgroup $G''$ of $\pi_1(\widetilde X^*,(x_0,\lambda))$ generated by (the homotopy classes in $(\widetilde X^*,(x_0,\lambda))$ of) all the elements of this form.

\begin{claim}\label{N-include-in-kernel}
The subgroups $G''$ and $G$ coincide.
\end{claim}

\begin{proof}
By Lemma \ref{lemma-fget}, to show that $G''\subseteq G$, we must prove that the loops (\ref{compfond}) are null-homotopic in $(\widetilde{X}, (x_0, \lambda))$. For loops of the form $\alpha^{-1} \cdot \widetilde h_{i}\circ \alpha$, this is proved in \cite[Lemma 4.8]{CE}. Now, since $(f\circ\alpha(0),\omega_i)$ is null-homotopic in $(\widetilde X,(f\circ\alpha(0),\lambda))$, it immediately follows that any loop of the form $\alpha^{-1} \cdot (f\circ\alpha(0),\omega_i)\cdot \widetilde h_{i}\circ \alpha$ is null-homotopic in $(\widetilde{X},(x_0, \lambda))$ too.

To prove that $G\subseteq G''$, we observe that any element of $G$ is written as a product of elements of the following form and their inverses:
\begin{align}\label{lgd}
[\beta]_{\widetilde X^*,(x_0,\lambda)}^{-1}\cdot[\alpha^{-1} \cdot 
(f\circ\alpha(0),\omega_i) \cdot \alpha]_{\widetilde X^*,(x_0,\lambda)} 
\cdot [\beta]_{\widetilde X^*,(x_0,\lambda)},
\end{align}
where $\beta\in F^1(\widetilde X^*,(x_0,\lambda))$, $\alpha\in F^1(\widetilde X_\lambda,\widetilde A_\lambda,(x_0,\lambda))$ and $1\leq i\leq N$. Clearly, any representative of the homotopy class (\ref{lgd}) is homotopic in $(\widetilde X^*,(x_0,\lambda))$ to the loop
\begin{align*}
\beta^{-1}\cdot(\alpha^{-1} \cdot 
(f\circ\alpha(0),\omega_i) \cdot (\widetilde h_i\circ\alpha\cdot (\widetilde h_i\circ\alpha)^{-1}) \cdot \alpha) \cdot \beta,
\end{align*}
which is the product of the following absolute loops:
\begin{align}\label{pfal}
\beta^{-1}\cdot
(\underbrace{\alpha^{-1} \cdot(f\circ\alpha(0),\omega_i) \cdot \widetilde h_i\circ\alpha}_{\mbox{\tiny loop in } F^1(\widetilde X^*,(x_0,\lambda))}) 
\cdot \beta\cdot\beta^{-1} \cdot (\underbrace{(\widetilde h_i\circ\alpha)^{-1} \cdot \alpha}_{\mbox{\tiny loop in } F^1(\widetilde X_\lambda,(x_0,\lambda))})
\cdot \beta.
\end{align}
Any product of homotopy classes in $(\widetilde X^*,(x_0,\lambda))$ of loops of the form (\ref{pfal}) and their inverses is an element of $G''$.
\end{proof}

We can now conclude the proof of Lemma \ref{lemma-fget-2}. By Claim \ref{N-include-in-kernel}, it suffices to show $G'=G''$. The inclusion $G'\subseteq G''$ is obvious. Conversely, any element of $G''$ is written as a product of elements of the following forms and their inverses:
\begin{align*}
[\beta]^{-1}\cdot
[\alpha^{-1} \cdot(f\circ\alpha(0),\omega_i) \cdot \widetilde h_i\circ\alpha] 
\cdot [\beta]
\quad\mbox{and}\quad
[\beta]^{-1} \cdot [\alpha^{-1} \cdot \widetilde h_i\circ\alpha]
\cdot [\beta],
\end{align*}
that is, by Claim \ref{sccp}, 
\begin{align*}
[\beta]^{-1}\cdot[(x_0,\omega_i)]\cdot[\beta]
\quad\mbox{and}\quad
[\beta]^{-1}\cdot [\alpha^{-1} \cdot \widetilde h_i\circ\alpha] \cdot [\beta],
\end{align*}
where $\beta\in F^1(\widetilde X^*,(x_0,\lambda))$, $\alpha\in F^1(\widetilde X_\lambda,\widetilde A_\lambda,(x_0,\lambda))$ and $1\leq i\leq N$, all the homotopy classes being in $\pi_1(\widetilde X^*,(x_0,\lambda))$.
The inclusion $G''\subseteq G'$ follows.
\end{proof}

\subsection{Conclusion}
Since the maps 
\begin{equation*}
\pi_1(\widetilde X_\lambda,(x_0,\lambda))\to\pi_1(X_\lambda,x_0)
\quad\mbox{and}\quad
\pi_1(\widetilde X,(x_0,\lambda))\to\pi_1(X,x_0)
\end{equation*}
(induced by the blowing up morphism $f$) are both isomorphisms (cf.~Section \ref{sect-bu}), Theorem \ref{mt} follows from Proposition \ref{fgxt}, Lemma \ref{commdiag-2} and the commutativity of diagram (\ref{commdiag-1}).

\bibliographystyle{amsplain}

\end{document}